\documentclass[12pt]{article}

\usepackage{amssymb}
\usepackage{amsmath} 
\usepackage{amsfonts}
 
\usepackage{amsthm}

\DeclareMathAlphabet{\mathpzc}{OT1}{pzc}{m}{it}

\theoremstyle{plain} 
\newtheorem{theorem}{Theorem}[section] 
\newtheorem{lemma}[theorem]{Lemma}

\newtheorem{question}[theorem]{Question} 
 
\newtheorem{proposition}[theorem]{Proposition}

\theoremstyle{definition}

\theoremstyle{remark}

\makeatletter
{\newcount\@hour}
{\newcount\@minute}
\def\timenow{\@hour=\time \divide\@hour by 60
\number\@hour:
  \multiply\@hour by 60 \@minute=\time
  \global\advance\@minute by -\@hour
  \ifnum\@minute<10 0\number\@minute\else
  \number\@minute\fi}
\def\ctimenow{\hfil{\tt \jobname.tex, \today~Time: \timenow }\hfil}
      \let\@oddfoot\ctimenow\let\@evenfoot\ctimenow
\makeatother

\begin{document}

\begin{center}{\bf \Large
Porosities of the sets of attractors
}
\end{center}
\smallskip
\begin{center}
By
\end{center}
\smallskip
\begin{center} Pawe\l{} Klinga and Adam Kwela
\end{center}

\begin{abstract}
This  paper  is  another  attempt  to  measure  the  difference between  the  family $A[0,1]$  of  attractors for iterated function systems acting on $[0,1]$ and  a  broader  family,  the  set $A_w[0,1]$ of attractors for weak iterated function systems acting on $[0,1]$. 

It is known that both $A[0,1]$ and $A_w[0,1]$ are meager subsets of the hyperspace $K([0,1])$ (of all compact subsets of $[0,1]$ equipped in the Hausdorff metric). Actually, $A[0,1]$ is even $\sigma$-lower porous while the question about $\sigma$-lower porosity of $A_w[0,1]$ is still open. 

We prove that $A[0,1]$ is not $\sigma$-strongly porous in $K([0,1])$. Moreover, we show that $A_w[0,1]\setminus A[0,1]$ is dense in $K([0,1])$. 
\let\thefootnote\relax
\let\thefootnote\relax
\end{abstract}

\section{Introduction}

We are mainly interested in two families of compact subsets of the space $[0,1]^d$, where $d\in\mathbb{N}$:
\begin{itemize}
\item the set $A[0,1]^d$ of attractors for iterated function systems acting on $[0,1]^d$,
\item the set $A_w[0,1]^d$ of attractors for weak iterated function systems acting on $[0,1]^d$.
\end{itemize}
We will consider both $A[0,1]^d$ and $A_w[0,1]^d$ as subsets of the hyperspace $K([0,1]^d)$ of all compact subsets of $[0,1]^d$ equipped in the Hausdorff metric. 

This approach is present in literature. In \cite[Theorem 3.9]{DS1} (see also \cite{DS2}) E. D'Aniello and T.H. Steele proved that the family $A[0,1]^d$ is a meager subset of $K([0,1]^d)$, for every $d\in\mathbb{N}$ (see also \cite[Theorem 5]{SS} where this fact is proved, however the statement of \cite[Theorem 5]{SS} is false). In \cite[Theorem 4.1]{DS3} and \cite[Theorem 4.1]{my}, independently, it is shown that the same holds for the family $A_w[0,1]^d$. Moreover, the result about attractors for classical iterated function systems has been strengthened in \cite[Theorem 3.1]{my} by showing that $A[0,1]^d$ is even smaller -- it is a $\sigma$-lower porous subset of $K([0,1]^d)$. The question about $\sigma$-lower porosity of $A_w[0,1]^d$ is still open (see \cite[Question 4.4]{my}).

There are still many other open problems concerning comparison of families $A[0,1]^d$ and $A_w[0,1]^d$. For instance, $A[0,1]^d$ is $F_\sigma$, while the exact descriptive complexity of $A_w[0,1]^d$ is unknown (recently in \cite{B} it only has been shown that it is not $F_{\sigma\delta}$). This  paper  is  another  attempt  to  measure  the  difference between  the  families $A[0,1]^d$ and $A_w[0,1]^d$.

In Section 2 we provide basic definitions needed in our considerations. Section 3 includes some basic preliminary results. For instance, we show that the set $A_w[0,1]^d\setminus A[0,1]^d$ cannot be nowhere dense itself as it is dense in $K([0,1]^d)$, and argue that $\sigma$-strong porosity is an appropriate notion for our purposes (as the collection of all finite subsets of $[0,1]^d$ is $\sigma$-strongly porous). In Section 4 we prove our main result, that $A[0,1]$ is not $\sigma$-strongly porous.

\section{Basic notions and definitions}

\subsection{Hyperspace of compact subsets}

Let $(X,d)$ be a compact metric space. By $K(X)$ we denote the family of all compact subsets of $X$. We will consider $K(X)$ as a topological space equipped with the Hausdorff metric $d_H$, which is given by 
$$d_H(A,B)=\min\left\{r\geq 0: A\subseteq\widetilde{B}_r\text{ and }B\subseteq\widetilde{A}_r\right\},$$
where for $C\in K(X)$ and $r\geq 0$ we write 
$$\widetilde{C}_r=\bigcup_{c\in C}\left\{x\in X: d(x,c)\leq r\right\}.$$
For $C\in K(X)$ and $r>0$ by $B_H(C,r)$ we will denote the ball in $K(X)$ of radius $r$ centered at $C$. Throughout the paper we will assume that $X=[0,1]^d$. For convenience in this case we will write $K[0,1]^d$ instead of $K([0,1]^d)$. Moreover, if $d=1$ then we will write $K[0,1]$ instead of $K[0,1]^1$.

For $r\in\mathbb{R}$, $x=(x_1,\ldots,x_d)\in[0,1]^d$ and $A\in K[0,1]^d$ we will write 
$$r\cdot A=\{(ry_1,\ldots,ry_d):(y_1,\ldots,y_d)\in A\},$$
$$x+A=\{(x_1+y_1,\ldots,x_d+y_d):(y_1,\ldots,y_d)\in A\}.$$ 
Note that both $rA$ and $x+A$ are compact subsets of $\mathbb{R}^d$. Thus, $rA\in K[0,1]^d$ and $x+A\in K[0,1]^d$ provided that $rA\subseteq[0,1]^d$ and $x+A\subseteq[0,1]^d$.

\subsection{Iterated function systems}

By contraction on $X$ we understand a function $f\colon X \to X$ such that:
$$\exists_{L\in [0,1)}\ \forall_{x,y\in X}\ d(f(x),f(y))\leqslant L\cdot d(x,y).$$
The smallest $L$ for which this inequality holds is called a Lipschitz constant of $f$ and denoted $Lip(f)$.

A function $f\colon X \to X$ is a weak contraction if for every pair of distinct points $x,y\in X$ it is true that $d(f(x),f(y))<d(x,y)$.

Every finite collection of contractions on $X$ will be called an iterated function system (IFS, in short) acting on $X$. Similarly, a finite collection of weak contractions on $X$ will be called a weak iterated function system (wIFS, in short) acting on $X$.

\subsection{Attractors}

Let $\{ s_1, \dots, s_k \}$ be an IFS acting on $X$. By $\mathcal{S}\colon K(X) \to K(X)$ we denote the Hutchinson operator for $\{ s_1, \dots, s_k \}$, i.e.
$$\mathcal{S}(A) = \bigcup_{i=1}^k s_i(A).$$
If $(X,d)$ is complete, then so is $(K(X),d_H)$. Therefore, by applying the Banach fixed-point theorem, the equation $\mathcal{S}(A) = A$ has a unique compact solution. Such set is called the attractor for the iterated function system $\{ s_1, \dots, s_k \}$. 

Analogously, for a wIFS $\{ s_1, \dots, s_k \}$ acting on $X$, a compact set $A\in K(X)$ satisfying $\bigcup_{i=1}^k s_i(A)=A$ is called a weak IFS attractor for the system $\{s_1,\dots,s_k\}$. It is shown in \cite{E} that for every weak iterated function system there exists a weak IFS attractor provided that $X$ is a compact space.

The form of the Hutchinson operator imposes that attractors are self-similar (at least in some sense), so they are often used to describe fractals. Clearly, an IFS attractor is a weak IFS attractor. However, the reverse inclusion is not true by \cite{NFM}.

By $A[0,1]^d$ and $A_w[0,1]^d$ we will denote the sets of all attractors for IFS and wIFS, respectively, acting on $[0,1]^d$. If $d=1$ then we will write $A[0,1]$ and $A_w[0,1]$ instead of $A[0,1]^1$ and $A_w[0,1]^1$, respectively.

\subsection{Porosities}

Fix a metric space $(X,d)$ and denote by $B(z,\delta)$ a ball in $(X,d)$ of radius $\delta>0$ centered at $z\in X$. For any subset $A$ of $X$ and any point $x\in X$ we define
$$\overline{p}(A,x)=\limsup_{R\to 0^+}\frac{\sup\{r\geq 0:B(y,r)\subseteq B(x,R)\setminus A\text{ for some }y\in X\}}{R},$$
$$\underline{p}(A,x)=\liminf_{R\to 0^+}\frac{\sup\{r\geq 0:B(y,r)\subseteq B(x,R)\setminus A\text{ for some }y\in X\}}{R}.$$
Observe that $\overline{p}(A,x)=1$ whenever $x\notin\overline{A}$ and $0\leq\overline{p}(A,x)\leq \frac{1}{2}$ in the opposite case (see \cite[Remark 1.2.(i)]{Preiss}).

We say that $A$ is:
\begin{itemize}
\item lower porous if $\underline{p}(A,x)>0$ for every $x\in X$,
\item upper porous if $\overline{p}(A,x)>0$ for every $x\in X$,
\item strongly porous if $\overline{p}(A,x)\geq \frac{1}{2}$ for every $x\in X$ (equivalently, $\overline{p}(A,x)=\frac{1}{2}$ for every $x\in \overline{X}$),
\item $\sigma$-lower porous ($\sigma$-upper porous, $\sigma$-strongly porous) if it is a countable union of lower porous (upper porous, strongly porous) sets.
\end{itemize}
Clearly, each lower porous set as well as each strongly porous set is upper porous. What is more, since the condition defining upper porosity is a strengthening of nowhere density, each $\sigma$-upper porous set is meager.

In this known that on the real line there are upper porous sets that are not $\sigma$-lower porous. In fact, there is even a closed strongly porous set which is not $\sigma$-lower porous (\cite[Corollary 4.1.(ii)]{shell}). On the other hand, there exist countable upper porous sets which are not strongly porous (see \cite[Remark 1.2(iv) and Example 4.6]{Preiss}). 

We will use a different, but equivalent (thanks to \cite[Remark 1.2.(ii)]{Preiss}), definition of strong porosity: $A$ is strongly porous, if for each $x\in X$ there are two sequences $(x_n)\subseteq X$ and $(r_n)\subseteq(0,+\infty)$ such that:
\begin{itemize}
\item $\lim_n x_n=x$,
\item $\lim_n\frac{r_n}{d(x,x_n)}=1$,
\item $B(x_n,r_n)\cap A=\emptyset$ for all $n$.
\end{itemize}

In \cite[Theorem 3.1]{my} it is shown that the set $A[0,1]^d$ is $\sigma$-lower porous. The question whether $A_w[0,1]^d$ is $\sigma$-lower porous is still open (see \cite[Question 4.4]{my}).

For more on porosities see for instance \cite{shell}, \cite{Preiss} or \cite{Z}.

\section{Preliminary results}

We start this section with a result enabling to transfer some results about attractors in $K[0,1]$ to arbitrary dimension $d\in\mathbb{N}$. However the proof can be found in \cite{D}, we repeat it for the sake of completeness.

\begin{lemma}[{\cite[Lemma 2.2]{D}}]
\label{l}
For each $d\in\mathbb{N}$ the following hold:
\begin{itemize}
\item $A\in A[0,1]$ if and only if $A\times\{0\}^{d-1}\in A[0,1]^d$,
\item $A\in A_w[0,1]$ if and only if $A\times\{0\}^{d-1}\in A_w[0,1]^d$.
\end{itemize}
\end{lemma}

\begin{proof}
Observe that if $f:[0,1]^d\to[0,1]^d$ is a (weak) contraction then so is $f\upharpoonright[0,1]\times\{0\}^{d-1}$. On the other hand, for each (weak) contraction $f:[0,1]\to[0,1]$ the map $g:[0,1]^d\to[0,1]^d$ given by $g(x_1,\ldots,x_d)=(f(x_1),0,\ldots,0)$ is a (weak) contraction as well. 
\end{proof}

Recall that $A[0,1]^d$ is dense in $K[0,1]^d$, for every $d\in\mathbb{N}$, as it contains all nonempty finite sets. In particular, $A[0,1]^d$ cannot be itself nowhere dense. The following observation implies that the same is true for the set $A_w[0,1]^d\setminus A[0,1]^d$. 

\begin{proposition}
The set $A_w[0,1]^d\setminus A[0,1]^d$ is dense in $K[0,1]^d$, for every $d\in\mathbb{N}$.
\end{proposition}

\begin{proof}
Before the main part of the proof, we need to perform a construction in $[0,1]$, which will be needed later. This part is a simple modification of \cite[Proposition 3.1]{D}, however we repeat the reasoning for the sake of completeness. 

Let $f:[0,1]\to[0,1]$ be the weak contraction given by $f(x)=x-x^2$ for all $x\in[0,1]$. Define a sequence of intervals by $I_1=[0,\frac{1}{2}]$ and $I_{n+1}=[0,f(\max I_n)]$. Inductively pick points $x_n$ for $n\in\mathbb{N}$ in such a way that for each $n\in\mathbb{N}$ we have:
\begin{itemize}
\item[(a)] $x_n\in I_1$,
\item[(b)] $x_{n+1}<x_n$,
\item[(c)] $x_n-x_{n+1}>x_{n+1}-x_{n+2}$,
\item[(d)] $\{x_i:i\in\mathbb{N}\}\cap(I_n\setminus I_{n+1})$ is finite,
\item[(e)] $k_n>n+n\cdot(\sum_{i=1}^{n-1}k_i)$, where $k_n=|\{x_i:i\in\mathbb{N}\}\cap(I_n\setminus I_{n+1})|$.
\end{itemize}
Note that item (d) guarantees that $\lim_n x_n=0$. Put $X=\{0\}\cup\{x_n:n\in\mathbb{N}\}$.

Observe that $X\in A_w[0,1]$. Indeed, let $g:[0,1]\to[0,1]$ be the weak contraction such that $g(x_n)=x_{n+1}$ for all $n\in\mathbb{N}$ (such $g$ exists by items (b) and (c)). As $\lim_n\max I_n=0$, the unique fixed point of $g$ is $0$. Thus, if $h:[0,1]\to[0,1]$ is the function constantly equal to $x_1$ then $X=g[X]\cup h[X]$. 

We are ready for the main part of the proof. As the family of all finite nonempty sets is dense in $K[0,1]^d$, it suffices to show that for every finite nonempty set $F\subseteq [0,1]^d$ and every $\delta>0$ there is an attractor in $B_H(F,\delta)\cap(A_w[0,1]^d\setminus A[0,1]^d)$.

Let $F$ be a finite nonempty subset of $[0,1]^d$ and $\delta>0$ be such that $2\delta<\min\{d(x,y):x,y\in F,x\neq y\}$. Fix any $w\in F$ and consider the set $Y=F\cup (w+\delta \cdot(X\times\{0\}^{d-1}))$. Clearly, $Y\in B_H(F,\delta)$. By Lemma \ref{l}, $w+\delta\cdot (X\times\{0\}^{d-1})\in A_w[0,1]^d$. Hence, $Y\in A_w[0,1]^d$ as $Y\setminus (w+\delta\cdot (X\times\{0\}^{d-1}))$ is finite. 

To finish the proof, we will show that $Y\notin A[0,1]^d$. Assume towards contradiction that $f_1,\ldots,f_k:[0,1]^d\to[0,1]^d$ are standard contractions such that $Y=\bigcup_{i=1}^k f_i[Y]$. Denote $X'=w+\delta \cdot(X\times\{0\}^{d-1})$ and $I'_n=w+\delta \cdot(I_n\times\{0\}^{d-1})$ for all $n\in\mathbb{N}$. Find $n\in\mathbb{N}$ such that: 
\begin{itemize}
\item[(i)] $n>k\cdot |F|$ (here $|F|$ denotes the cardinality of the finite set $F$),
\item[(ii)] if $f_i(w)\neq w$ then $f_i[X']\cap X'\cap I'_n=\emptyset$ (notice that $f_i(w)\neq w$ implies $f_i[X']\cap X'$ being finite, by $\lim_n x_n=0$),
\item[(iii)] $\max_{i\leq k}Lip(f_i)\cdot|I_n|<|I_{n+1}|$ (recall that for $f=x(1-x)$ and any $L\in(0,1)$ we have $|f(x)-f(0)|=|x|\cdot|1-x|>L|x-0|$ for all $x\in(0,1-L)$).
\end{itemize}
We will justify that $X'\cap I'_n\setminus I'_{n+1}$ cannot be covered by $\bigcup_{i\leq k}f_i[Y]$. If $f_i(w)\neq w$ then $f_i[X']\cap X'\cap I'_n=\emptyset$ by item (ii). If $f_i(w)=w$ then item (iii) implies that $f_i[I'_n]\subseteq I'_{n+1}$. Therefore, 
$$(X'\cap I'_n\setminus I'_{n+1})\cap\bigcup_{i\leq k}f_i[Y]=(X'\cap I'_n\setminus I'_{n+1})\cap\bigcup_{i\leq k}f_i[F\cup(X'\cap I'_1\setminus I'_n)].$$ 
Using $|X'\cap I'_1\setminus I'_n|=\sum_{i=1}^{n-1}k_i$, we see that 
$$\left|\bigcup_{i\leq k}f_i[F\cup(X'\cap I'_1\setminus I'_n)]\right|\leq k\cdot\left(|F|+\sum_{i=1}^{n-1}k_i\right)<$$
$$<n+n\cdot\sum_{i=1}^{n-1}k_i<k_n=|X'\cap(I'_n\setminus I'_{n+1})|.$$ 
Thus, the proof is finished.
\end{proof}

We end with showing that $\sigma$-strong-porosity is an appropriate notion for our purposes -- settling whether $A[0,1]^d$ is $\sigma$-strongly-porous will require full depth of $A[0,1]^d$ (not only the information that all finite nonempty sets are in $A[0,1]^d$).

\begin{proposition}
The collection of all finite nonempty subsets of $[0,1]^d$ is $\sigma$-strongly-porous in $K[0,1]^d$, for every $d\in\mathbb{N}$.
\end{proposition}

\begin{proof}
Observe that the family of all finite nonempty subsets of $[0,1]^d$ is equal to $\mathcal{P}(D)\cup\bigcup_{k\in\omega}\mathcal{F}_k$, where
$$D=\left\{(x_1,\ldots,x_d)\in[0,1]^d:\forall_{i\leq d}\ x_i=0\text{ or }x_i=1\right\},$$
$$\mathcal{F}_k=\{F\subseteq[0,1]^d: |F|=k,F\setminus D\neq\emptyset\}.$$
Clearly, $\mathcal{P}(D)$ is strongly porous as a finite set. Thus it suffices to show that $\mathcal{F}_k$ is strongly porous for each $k\in\mathbb{N}$. 

Fix $k\in\mathbb{N}$, $F\in \mathcal{F}_k$ and any $x\in F\setminus D$. Find three sequences $(y_n),(z_n)\subseteq[0,1]^d$ and $(r_n)\subseteq(0,1)$ such that:
\begin{itemize}
\item $r_n=d(y_n,x)=d(z_n,x)$ for all $n$,
\item $y_n\neq z_n$ for all $n$,
\item $B(y_n,r_n)\cap B(z_n,r_n)=\emptyset$ for all $n$,
\item $\lim_n r_n=0$.
\end{itemize}
This is possible as $x\notin D$.

Put $G_n=(F\setminus\{x\})\cup\{y_n,z_n\}$. As $d_H(G_n,F)=r_n\to 0$, the sequence $(G_n)$ tends to $F$ (in $K[0,1]^d$) and $\lim_n\frac{r_n}{d_H(G_n,F)}=1$. To conclude the proof, we need to justify that $B_H(G_n,r_n)\cap\mathcal{F}_k=\emptyset$.

Let $H\in B_H(G_n,r_n)$. Then $H$ has to intersect each of the balls $B(y_n,r_n)$, $B(z_n,r_n)$ and $B(w,r_n)$ for $w\in F\setminus\{x\}$. As those balls are pairwise disjoint, $H$ has to have at least $2+k-1=k+1$ elements. Hence, $H\notin \mathcal{F}_k$ and the proof is finished.
\end{proof}

\section{Attractors for IFSs}

\begin{theorem}
\label{strong}
The set $A[0,1]$ is not $\sigma$-strongly porous in $K[0,1]$.
\end{theorem}

\begin{proof}	
	We will prove the following.
	$$ \exists_{E \in A[0,1]}\:\exists_{\varepsilon>0}\:\forall_{Y\in B_H(E,\varepsilon)} \:\: B_H(Y, d_H(Y,E)\cdot 0.99)\cap A[0,1] \neq \emptyset. $$
	From this, the negation of the definition of $\sigma$-strong porosity will follow.
	
	The idea of the proof is to find a correspondence between $A[0,1]$ and the set $[-\frac{3}{2},\frac{3}{2}]^\mathbb{N}$ of all sequences with all terms belonging to the interval $[-\frac{3}{2},\frac{3}{2}]$, then observe that in any decomposition of $[-\frac{3}{2},\frac{3}{2}]^\mathbb{N}$ into countably many pieces, one will be "large" (i.e. dense in some set $U\subseteq [-\frac{3}{2},\frac{3}{2}]^\mathbb{N}$ open in the product topology), pick any attractor $E$ corresponding to an element of $U$ and show that for any $Y\in B_H(E,\varepsilon)$ there is an open $V\subseteq U$ such that the attractor corresponding to any element of $V$ is in $B_H(Y, d_H(Y,E)\cdot 0.99)$.
	
	The correspondence will be as follows: we will match $(x_n)\in[-\frac{3}{2},\frac{3}{2}]^\mathbb{N}$ (i.e. a sequence of reals) with a Cantor-like attractor. By a Cantor-like attractor we mean an attractor generated by two contractions: each of them will be "based" on the following functions:
	$$f_1(x) = \frac{1}{10}x + \frac{2}{10}, \:\: f_2(x) = \frac{1}{10}x + \frac{7}{10}.$$
	Therefore, in each iteration a new image will be a "tenth" of a size of the previous image. However, also, every image of an interval can be moved by $\frac{3}{2}$ of the tenth of the interval to the left or to the right (for instance, in the case of the first iteration: the left-hand image will be a subset of $[\frac{1}{20},\frac{9}{20}]$ and the right-hand image will be a subset of $[\frac{11}{20},\frac{19}{20}]$ -- note that the distance between them is at least $\frac{1}{10}$, so at the end we will still get an attractor for some IFS). The endpoints of those images will be determined by the sequence $(x_n)$. For instance, a sequence constantly equal to 0 will "code" exactly the functions $f_1, f_2$. A sequence constantly equal to $-\frac{3}{2}$ will result in functions 
	$$\frac{1}{10}x + \frac{1}{20}, \frac{1}{10}x + \frac{11}{20}$$ 
	and a sequence constantly equal to $\frac{3}{2}$ will result in functions 
	$$\frac{1}{10}x + \frac{7}{20}, \frac{1}{10}x + \frac{17}{20}.$$ 
	And obviously, all the "in-between" shifts of intervals will be coded by terms between $-\frac{3}{2}$ and $\frac{3}{2}$. Below we present details.
	
	Let the correspondence between $(x_n)$ and $(f_1,f_2)$ be the following: for each $i\in\mathbb{N}$ let $x_i$ be such that:
	\begin{itemize}
		\item[$x_1:\:\:\:\:$] $f_1(0) = \frac{2}{10} + \frac{1}{10}\cdot x_1$
		\item[$x_2:\:\:\:\:$] $f_2(0) = \frac{7}{10} + \frac{1}{10}\cdot x_2$
		\item[$x_3:\:\:\:\:$] $f_1(f_1(0)) = f_1(0) + \frac{2}{100} + \frac{1}{100} \cdot x_3$
		\item[$x_4:\:\:\:\:$] $f_2(f_1(0)) = f_1(0) + \frac{7}{100} + \frac{1}{100} \cdot x_4$
		\item[$x_5:\:\:\:\:$] $f_1(f_2(0)) = f_2(0) + \frac{2}{100} + \frac{1}{100} \cdot x_5$
		\item[$x_6:\:\:\:\:$] $f_2(f_2(0)) = f_2(0) + \frac{7}{100} + \frac{1}{100} \cdot x_6$
		\item[$x_7:\:\:\:\:$] $f_1(f_1(f_1(0))) = f_1(f_1(0)) + \frac{2}{1000} + \frac{1}{1000} \cdot x_7$
	\end{itemize}
and so on.
	
Suppose to the contrary that $A[0,1]$ is $\sigma$-strongly porous. Then so is the family $\mathcal{F}$ of attractors corresponding (in the above way) to sequences in $[-\frac{3}{2},\frac{3}{2}]^\mathbb{N}$. In particular, $\mathcal{F}=\bigcup_{m\in\mathbb{N}} P_m$ where each $P_m$ is a strongly porous subset of $K[0,1]$. The sets $P_m$ define a cover of $[-\frac{3}{2},\frac{3}{2}]^\mathbb{N}$ by countably many pieces. Since $[-\frac{3}{2},\frac{3}{2}]^\mathbb{N}$ is a complete metric space, by the Baire category theorem there is $m\in\mathbb{N}$ and an open set $I_1\times I_2 \times \dots \times I_k \times [-\frac{3}{2}, \frac{3}{2}]^\mathbb{N}\subseteq [-\frac{3}{2}, \frac{3}{2}]^\mathbb{N}$ (where $I_1,\ldots,I_k$ are open intervals in $[-\frac{3}{2}, \frac{3}{2}]$) such that the sequences corresponding to attractors from $P_m$ are dense in $I_1\times I_2 \times \dots \times I_k \times [-\frac{3}{2}, \frac{3}{2}]^\mathbb{N}$. For every $i\leq k$ put $x_i$ as a middle of an interval $I_i$. Then the sequence $(x_1, x_2, \dots, x_k, 0, 0, \dots )\in[-\frac{3}{2}, \frac{3}{2}]^\mathbb{N}$ determines an attractor $E\in K[0,1]$.

Let us take $\varepsilon$ such that it is half the length of an interval that is one level "below" $k$-th interval in the construction of the attractor. More formally, since $x_1$ and $x_2$ are responsible for the positions of intervals of length $\frac{1}{10}$, $x_3,x_4,x_5$ and $x_6$ are responsible for the positions of intervals of length $\frac{1}{100}$ etc., $\varepsilon = \frac{1}{2\cdot 10^n}$, where $n\in\mathbb{N}$ is such that $\sum_{i=1}^{n-2}2^i<k\leq \sum_{i=1}^{n-1}2^i$. 

Now fix $Y\in B_H(E,\varepsilon)$ and to shorten the notation put $\delta = d_H(Y,E)$. Obviously, $\delta < \varepsilon= \frac{1}{2\cdot 10^n}$. Thus, there exists $j\geq n$ such that $\delta \in [\frac{1}{2}\cdot\frac{1}{10^{j+1}}, \frac{1}{2}\cdot\frac{1}{10^j} )$. Let $J_1,\ldots,J_{2^{j+2}}$ be the intervals of length $\frac{1}{10^{j+2}}$ from the construction of the attractor $E$. 

Observe that:
\begin{itemize}
\item $\widetilde{Y}_\delta\supseteq E$ (since $\delta = d_H(Y,E)$),
\item $\delta-\frac{99}{100}\delta=\frac{\delta}{100}<\frac{1}{2\cdot 10^{j+2}}$,
\item the maximal possible shift of each $J_i$ is equal to 
$$\frac{3}{2\cdot 10^{j+2}}=\frac{1}{10^{j+2}}+\frac{1}{2\cdot 10^{j+2}}>|J_i|+\frac{\delta}{100},$$
\item the set $\widetilde{Y}_{\frac{99}{100}\delta}$ is a union of closed intervals each of which has length at least $2\frac{99}{100}\delta\geq\frac{99}{10^{j+3}}>\frac{1}{10^{j+2}}=|J_i|$.
\end{itemize}
Thus, for each $i\leq 2^{j+2}$ the set $U_i\subseteq[-\frac{3}{2}, \frac{3}{2}]$ equal to the interior of
$$\left\{x\in\left[-\frac{3}{2}, \frac{3}{2}\right]:J_i+x\subseteq\widetilde{Y}_{\frac{99}{100}\delta}\right\}$$
is nonempty. Hence, there is an open subset of $I_1\times I_2 \times \dots \times I_k \times [-\frac{3}{2}, \frac{3}{2}]^\mathbb{N}$ such that each attractor corresponding to an element of that open set is a subset of $\widetilde{Y}_{\frac{99}{100}\delta}$. 

Actually, by the last item above, we can even conclude that the set $U_i\setminus[-10^j\delta,10^j\delta]$ is nonempty for each $i\leq 2^{j+2}$. Indeed, the interval $[-10^j\delta,10^j\delta]$ is responsible for shifts by at most $\frac{\delta}{100}<\frac{1}{2}\frac{1}{10^{j+2}}$ (as each $x\in[-\frac{3}{2}, \frac{3}{2}]$ corresponds to a shift of $J_i$ by $\frac{x}{10^{j+2}}$). Thus, it suffices to observe that $|J_i|=\frac{1}{10^{j+2}}$, $|[-\frac{1}{100}\delta,\frac{1}{100}\delta]|<\frac{1}{10^{j+2}}$ and $2\frac{99}{100}\delta> 2\frac{1}{10^{j+2}}$.

Now, for each $i\leq 2^{j+2}$ let $V_i\subseteq U_i$ be open and such that:
\begin{itemize}
\item if $i$ is odd (i.e., the closest interval $J_{i'}$ to $J_i$ is on the right-hand side of $J_i$) and $U_i\cap[-\frac{3}{2},-10^j\delta)\neq\emptyset$ then $V_i=U_i\cap[-\frac{3}{2},-10^j\delta)$,
\item if $i$ is even (i.e., the closest interval $J_{i'}$ to $J_i$ is on the left-hand side of $J_i$) and $U_i\cap(10^j\delta,\frac{3}{2}]\neq\emptyset$ then $V_i=U_i\cap(10^j\delta,\frac{3}{2}]$,
\item $V_i=U_i\setminus [-10^j\delta,10^j\delta]$ in all other cases.
\end{itemize}

Let $H\in P_m$ be an attractor corresponding to any sequence in 
$$I_1\times I_2 \times \dots \times I_k \times \left[-\frac{3}{2}, \frac{3}{2}\right]^{\sum_{i=1}^{j+1}2^i-k}\times V_1\times\ldots\times V_{2^{j+2}}\times \left[-\frac{3}{2}, \frac{3}{2}\right]^\mathbb{N}.$$
Then $H\subseteq \widetilde{Y}_{\frac{99}{100}\delta}$. Moreover, $Y\subseteq \widetilde{H}_{\frac{99}{100}\delta}$. Indeed, fix $y\in Y$ and let $e\in E$ be such that $d(y,e)$ is minimal possible (i.e., $d(y,e)=d(y,E)$). Without loss of generality we may assume that $e\in J_i$ for some odd $i\leq 2^{j+2}$ (the case of even $i$ is similar). Let $z$ be the point lying exactly in the middle between $J_i$ and $J_{i+1}$. Denote also by $e'$ the point corresponding to $e$ in $H$. Then $s=d(e,e')$ is the shift of $J_i$. There are three possibilities: 
\begin{itemize}
\item $y\leq e'$ and $J_i$ was shifted to the left (i.e., towards $y$). In this case 
$$d(y,e')=d(y,e)-s<\delta-\frac{\delta}{100}=\frac{99}{100}\delta$$
as the shift was by more than $\frac{\delta}{100}$.
\item $J_i$ was shifted to the left, but $y>e'$. In this case the most extreme possibility is $y=z$. However, we have:
$$d(z,e)\leq|J_i|+\frac{1}{2}\text{dist}(J_i,J_{i+1})=\frac{1}{10^{j+2}}+\frac{2}{10^{j+2}}=\frac{3}{10^{j+2}}.$$ 
Since $s\leq\frac{3}{2}\cdot\frac{1}{10^{j+2}}$, we get that
$$d(z,e')=d(z,e)+s\leq\frac{3}{10^{j+2}}+\frac{3}{2}\cdot\frac{1}{10^{j+2}}<\frac{9}{2}\cdot\frac{1}{10^{j+2}}<\frac{99}{100}\frac{1}{2}\frac{1}{10^{j+1}}\leq\frac{99}{100}\delta.$$
\item $J_i$ was shifted to the right. We will show that this implies that there is no $y\in Y$ such that $d(y,e)=d(y,E)$ for $e\in J_i$, i.e., this case is impossible. Indeed, if $y>z$ then $d(y,J_{i+1})<d(y,J_{i})$. On the other hand, if $y<\min J_i-\delta$ then $d(y,J_{i-1})<d(y,J_{i})$ (as $y\in\bigcup_{i'\leq 2^{j+2}}\widetilde{(J_{i'})}_\delta$). Finally, if $\min J_i-\delta\leq y\leq z$ then it would be possible to shift $J_i$ by some $\frac{\delta}{100}<t<\frac{3}{2}\frac{1}{10^{j+2}}$ to the left in such a way that $t+J_i\subseteq B(y,\frac{99}{100}\delta)$ as $2\frac{99}{100}\delta>|J_i|$ and
$$d(z,\min J_i)+\frac{\delta}{100}\leq\frac{3}{10^{j+2}}+\frac{1}{2}\frac{1}{100}\frac{1}{10^{j}}<\frac{99}{100}\frac{1}{2}\frac{1}{10^{j+1}}\leq\frac{99}{100}\delta,$$
$$d(\min J_i-\delta,\max J_i)-\frac{3}{2}\frac{1}{10^{j+2}}=\delta+|J_i|-\frac{3}{2}\frac{1}{10^{j+2}}=\delta-\frac{1}{2}\frac{1}{10^{j+2}}<\frac{99}{100}\delta,$$
(recall that $\delta<\frac{1}{2}\frac{1}{10^{j}}$).
\end{itemize}
This finishes the proof.
\end{proof}

Knowing the status of $\sigma$-strong porosity of $A[0,1]$, we ask the same question for the difference between this family and a broader one, namely the family of weak attractors.

\begin{question}
Is the set $A_w[0,1]\setminus A[0,1]$ $\sigma$-strongly-porous in $K[0,1]$?
\end{question}

\newcommand{\nosort}[1]{}

\begin{center}
\flushleft{{\sl Addresses:}} \\
Pawe\l{} Klinga \\
Institute of Mathematics \\
University of Gda\'nsk \\
Wita Stwosza 57 \\
80 -- 952 Gda\'nsk \\
Poland\\
e-mail: pawel.klinga@ug.edu.pl
\end{center}

\begin{center}
\flushleft{{\sl Address:}} \\
Adam Kwela \\
Institute of Mathematics \\
University of Gda\'nsk \\
Wita Stwosza 57 \\
80 -- 952 Gda\'nsk \\
Poland\\
e-mail: adam.kwela@ug.edu.pl
\end{center}

\end{document}